\documentclass[preprint, 12pt]{article}%elsarticle}

\pdfoutput=1

\usepackage{arxiv}

\usepackage{amsmath,amstext,amssymb}
\usepackage{graphicx,bm,color}
\usepackage{amsthm}
\usepackage{mathtools}
\usepackage{ulem} 
\usepackage{subcaption}
\usepackage{hyperref}
\usepackage{soul}
\usepackage[nottoc]{tocbibind}
\usepackage[affil-it]{authblk}
\usepackage[utf8]{inputenc} % allow utf-8 input

\usepackage{float}%To place a float exactly at a point (using H instead of !ht)
\usepackage{setspace}   %Allows double spacing with the
\DeclarePairedDelimiter{\ceil}{\lceil}{\rceil}
\DeclarePairedDelimiter{\floor}{\lfloor}{\rfloor}
\newtheorem{theorem}{Theorem}[section]
\newtheorem{corollary}[theorem]{Corollary}
\newtheorem{lemma}[theorem]{Lemma}

\def\keyword{\vspace{.5em}{\textbf{Keywords}.\,\relax}}

\title{Inverse properties of a class of pentadiagonal matrices related to higher order difference operators}

\author[1]{Bakytzhan Kurmanbek}
\author[2]{Yogi Erlangga}

\author[1,*]{Yerlan Amanbek}%}\thanks{Corresponding author: yerlan.amanbek@nu.edu.kz}}

\affil[1]{Nazarbayev University, Department of Mathematics, 53 Kabanbay Batyr Ave, Nur-Sultan 010000, Kazakhstan}
\affil[2]{Zayed University, Department of Mathematics, Abu Dhabi Campus, P.O. Box 144534, United Arab Emirates}
\affil[ ]{\textit{ bakytzhan.kurmanbek@nu.edu.kz, yogi.erlangga@zu.ac.ae, yerlan.amanbek@nu.edu.kz}}

\affil[*]{Corresponding author}
\begin{document}
%\begin{frontmatter}
\maketitle

% \linespread{0.5}
% \author[2]{Bakytzhan Kurmanbek\fnref{fn2}}
% \ead{bakytzhan.kurmanbek@nu.edu.kz}

% \author[1]{Yogi A. Erlangga%\fnref{fn1}
% }
% \ead{yogi.erlangga@zu.ac.ae}

% \author[2]{Yerlan Amanbek\corref{cor1}}
% \ead{yerlan.amanbek@nu.edu.kz}

% %\affil[ ]{\textit{yerlan.amanbek@nu.edu.kz, \{gurpreet, gergina \}@utexas.edu, mfw@ices.utexas.edu}}

% %\fntext[fn1]{This is the first author footnote.}

% %\fntext[fn3]{Yet another author footnote.}
% \address[1]{Zayed University, Department of Mathematics, Abu Dhabi Campus, P.O. Box 144534, United Arab Emirates}
% \address[2]{Nazarbayev University, Department of Mathematics, 53 Kabanbay Batyr Ave, Nur-Sultan 010000, Kazakhstan}
% \cortext[cor1]{Corresponding author}
% \fntext[fn2]{Part of this work was made by this author during his visit to the first author, financially supported by the Nazarbayev University Faculty Development Competitive Research Grant No. 110119FD4502}

\fontsize{11pt}{11pt}\selectfont

\begin{abstract}
This paper analyzes the convergence of fixed-point iterations of the form $\boldsymbol{u} = f(\boldsymbol{u})$ and the properties of the inverse of the related pentadiagonal matrices,  associated with the fourth-order nonlinear beam equation. This nonlinear problem is discretized using the finite difference method with the clamped-free and clamped-clamped boundary conditions in the one dimension. Explicit formulas for the inverse of the matrices and norms of the inverse are derived. In iterative process, the direct computation of inverse matrix allows to achieve an efficiency. Numerical results were provided.
\end{abstract}

\keyword{explicit formula, pentadiagonal matrices, finite difference, nonlinear beam equation, fixed point method}
% \begin{keyword}
% explicit formula \sep pentadiagonal matrices \sep finite difference \sep nonlinear beam equation \sep Fixed point method
% \end{keyword}
%\end{frontmatter}

% \fontsize{11pt}{11pt}\selectfont	

\section{Introduction} \label{section.intro}
	
Many applications give arise to mathematical problems that involve numerical computations with pentadiagonal matrices, which require their inversion (see~\cite{Wang15LZ} and references therein). Even though inversion of a nonsingular pentadiagonal matrix can be done efficiently by a numerical linear algebra software, explicit inverse formulas are useful, for example, in a computer algebra software.

Early results on inverses of banded matrices can be traced as far back as to the work of~\cite{Allgower73, Barrett81F, Rehnqvist72} for general band matrices. Results for band Toeplitz matrices are given in~\cite{Dow03}, with explicit inverse formulas for tridiagonal matrices in~\cite{Barrett79,Jia13SE} and pentadiagonal matrices in~\cite{Rozsa87,Zhao08,Diele98L,Wang15LZ,Lv07H,kurmanbek2021inverse,elouafi2018explicit}. In addition, properties including determinants of such matrices related to finite difference operators have been investigated, e.g. in~\cite{amanbek2020explicit,kurmanbek2020proof,shitov2021determinants}.

In this study, we focus on the specific pentadiagonal matrices arising in a fixed-point iteration for numerically solving  the fourth-order nonlinear beam equation:
\begin{eqnarray}
	  \displaystyle \frac{d^4 \widehat{\phi}}{d\widehat{x}^4} = \alpha_1 e^{\displaystyle -\alpha_2 \widehat{\phi}},  \quad  \widehat{x} \in \Omega = (0,L). \notag 
\end{eqnarray}
This nonlinear equation finds applications in mechanical and civil engineering, which models, e.g., a cantilever beam subjected to swelling pressure on one side. In the above equation, the right-hand side term is the swelling pressure, which in this form is proposed by Grob~\cite{Grob72}, based on empirical studies (see, e.g.,~\cite{Wolf03F} and the references therein),  $L > 0$ is the length of the beam, and $\alpha_1, \alpha_2 > 0$ represents the mechanical property of the beam, which are assumed to be constant. 

Scaling the domain to unity using the dimensionless variable $x = \widehat{x}/L$ and setting $\phi = \alpha_2 \widehat{\phi}$ yields
    \begin{eqnarray}
    \displaystyle \frac{d^4 \phi}{d x^4} = K e^{-\phi}, \text{ in } \Omega = (0,1),  \label{eq1}
    \end{eqnarray}
where $K = \alpha_1 \alpha_2  L > 0$. We shall use this formulation throughout. For (\ref{eq1}) two types of boundary conditions are employed:
    \begin{enumerate}
    \item Clamped-Free (CF) condition:
    \begin{eqnarray}
      \phi(0) = \phi'(0) = 0  \quad \text{and} \quad \phi''(1) = \phi'''(1) = 0,  \label{eq2}
    \end{eqnarray}
    \item Clamped-Clamped (CC) condition:
    \begin{eqnarray}
    \phi(0) = \phi'(0) = 0  \quad \text{and} \quad \phi(1) = \phi'(1) = 0.  \label{eq3}
    \end{eqnarray}
    \end{enumerate}
Since $\displaystyle \frac{d^4 \phi}{d x^4} = K e^{-\phi} > 0$, obviously, $\phi = 0$ can not be a solution, even though it satisfies the boundary conditions. 

The solution of \eqref{eq1} with the boundary conditions \eqref{eq2} is concave up and an increasing function, which can be deduced from a mixed formulation of \eqref{eq2}:
    \begin{equation}
      \begin{cases}
          \displaystyle \frac{d^2 \omega}{dx^2} = K e^{-\phi},&\omega(1) = \omega'(1) = 0, \\[10pt]
          \displaystyle \frac{d^2 \phi}{dx^2} =   \omega,& \phi(0) = \phi'(0) = 0. \\
      \end{cases} \label{eq4}
    \end{equation}
From the first part of \eqref{eq4}, with $e^{-\phi} > 0$ in $\Omega$, $\omega'' > 0$, and $w'$ increases in $\Omega$. The condition $\omega'(1) = 0$ requires that $w' < 0$ in $\Omega$, which furthermore, together with the condition $\omega(1)=0$, implies that  $\omega > 0$ and decreases. From the second part of \eqref{eq4}, we have $\phi'' = \omega > 0$; thus, $\phi'$ is an increasing function in $\Omega$. Since $\phi'(0) = 0$, $\phi > 0$, which implies $\phi > 0$ and increases. This characterization also holds in the finite-difference setting based on the second-order scheme we use in this paper (c.f. Section~\ref{section.num}).

Numerical methods based on finite element methods for~\eqref{eq1} have been proposed and studied, e.g., in~\cite{Skrzypacz19BNK,Wei16LZKK}, where focus is given on the accurate approximation of the solution. This paper approaches the problem from a different angle, with emphasis put on the convergence of the iteration method of the form
    $$
      \phi = \mathcal{L}^{-1}  \left( Ke^{-\phi} \right),
    $$
where $\mathcal{L} = d^4/dx^4$, and the properties of the related iteration matrices involved. Using the second-order finite difference approach, these matrices are pentadiagonal and near Toeplitz. 

In this paper, we present explicit formulas for inverses of the specific pentadiagonal matrices and their bounds of norms, which are necessary in the convergence analysis of the fixed-point iteration. As the inverse can be formed explicitly, we are able to construct an exact norm of some of those matrices. The convergence rate for the clamped-free and clamped-clamped problems were derived and then numerical examples were presented for different parameters. 

% Relevant to the inverse of the specific pentadiagonal matrices considered in this paper is the convergence analysis of the fixed-point iteration, which requires an estimate of the norm of the inverse matrix. As the inverse can be formed explicitly, we are able to construct an exact norm of some of those matrices. 

%%%%

The paper is organized as follows. Section~\ref{section.CF} is devoted to the convergence and the inverse of the iteration matrix for problem with the clamp-free condition. Similar discussion for the clamp-clamp condition is given in Section~\ref{section.CC}. Numerical results are presented in Section~\ref{section.num}, followed by some concluding remarks in Section~\ref{section.conclusion}.

\section{The case with clamped-free boundary conditions} \label{section.CF}
We consider $n+1$ equidistant grid points on the closed interval $[0,1]$, with the distance (grid size) $h = 1/n$, at which the solution of \eqref{eq1} is approximated by a finite difference scheme. Each grid point is indexed by $i = 0,\dots,n$, where $i = 0$ and $n$ correspond to the boundary points.  Throughout the paper, we shall consider $n \ge 5$ for $A$ to be a meaningful approximation to the differential operator $\mathcal{L}$, even though $n = 5$ may not be of practical interest. 
    
For the interior nodes, $1 \le i \le n-1$, the fourth-order derivative is approximated by the second-order finite difference scheme:
    $$
      \frac{d^4 \phi}{dx^4} (x_i) \approx \frac{1}{h^4} (\phi_{i-2} - 4\phi_{i-1} + 6 \phi_i - 4 \phi_{i+1} + \phi_{i+2}),
    $$
where $x_i = ih$ and $\phi_i \equiv \phi(x_i)$. For $i =2$, we just impose the boundary condition $\phi(0) \equiv \phi_0 = 0$. For $i = 1$, $\phi_{-1}$ corresponds to a fictitious point outside the computational domain, which is eliminated using the central scheme approximation to the boundary condition $\phi'(0) = 0$. Similar approaches are used for $i = n-1$ and $n$, with the boundary conditions $\phi''(1) = \phi'''(1) = 0$ be approximated by appropriate second-order finite difference schemes.
    
The resultant system of nonlinear equations is
    \begin{eqnarray}
       A \boldsymbol{u} = h^4 K \exp(\boldsymbol{-u}),  \label{eq5}
    \end{eqnarray}
    where $\boldsymbol{u} = (u_1, \dots, u_n)^T \in \mathbb{R}^n$, with $u_i \approx \phi(x_i)$, and
    \begin{eqnarray}
      A := \left[ \begin{array}{rrrrrrr} %{ccccccccccc}
            7     &   -4   &   1   &   0   & \cdots &  & 0 \\
           -4     &    6   &  -4   &  \ddots  & \ddots  &   & \vdots  \\
            1     &   -4   &  \ddots & \ddots &   &   &  \\
           0 & \ddots &  \ddots   &   & \ddots & 1 & 0 \\  
                   & \ddots  &     &  \ddots & 6 & -4 & 1 \\      
         \vdots &          &  &  1 & -4 & 5 & -2 \\
         0        &\cdots   & &  0 & 2 & -4 & 2
      \end{array}\right]. \label{eq6}
    \end{eqnarray}
Here, $A \in \mathbb{R}^{n \times n} $ is a nonsymmetric, nondiagonally dominant pentadiagonal matrix. 

Our first result on $A$ is that it is nonsingular. In fact, %\sout{with $A = [a_{ij}]_{i,j = 1,n}$}
we have the following theorem of the explicit inverse of matrix
 \begin{theorem} \label{theo01}
    	Let $B = [b_{i,j}]_{i,j = 1,n} \in \mathbb{R}^{n \times n}$ such that
    	\begin{align*}
    	   b_{i,j} &= \displaystyle \frac{3ij^2 + j - j^3}{6}, \quad \forall j \le i, i \in \{1,2,\dots,n\}, j \in \{1,2,\dots,n-1\}, \notag \\[.5em]
    	   b_{i,n} &= \frac{1}{2}b_{n,i}, \notag \\[.5em]
    	   b_{n,n} &= \frac{1}{12} n(2n^2+1), \notag \\[.5em]
    	   b_{i,j} &= b_{j,i}, \quad i,j \in \{1,2,\dots,n-1\}. \notag
\end{align*}
Then $B$ is the inverse of $A$, where $A = [a_{i,j}]_{i,j = 1,n}$ is given in (\ref{eq6}).
\end{theorem}
\begin{proof}
The proof is done by the direct computation. % and is shown only for the case $3 \le i \le n-2$ and $1 \le j \le n$ ({ \bf for checking complete proof is shown below}). 
Let $D$ be matrix such that $D=AB$. We want to show that the product 
\begin{equation*}
d_{i, j} := \left[a_{i,1} \,\, a_{i,2} \,\, \cdots \,\,a_{i,n}\right] \left[ \begin{array}{c} b_{1,j} \\ b_{2,j} \\ \vdots \\ b_{n,j} \end{array} \right] = \begin{cases} 1,& i = j, \\ 0, & i \neq j. \end{cases} 
\end{equation*}
In other words, $D$ is the identity matrix $n \times n$.
%d := \left[a_{i1} \,\, a_{i2} \,\, \cdots \,\,a_{in}\right] \left[ \begin{array}{c} b_{1j} \\ b_{2j} \\ \vdots \\ b_{nj} \end{array} \right] = \begin{cases} 1,& i = j, \\ 0, & i \neq j. \end{cases}
%$$
\begin{enumerate}
\item[(i)] The case $3 \le i \le n-2$ and $1 \le j \le n$. 

In this case, $a_{i,i-2} = 1$, $a_{i,i-1} = -4$, $a_{i,i} = 6$, $a_{i,i+1} = -4$, $a_{i,i+2} = 1$, while the others are $0$. Therefore, 
\begin{equation}
    d_{i, j} = b_{i-2,j} - 4 b_{i-1,j} + 6 b_{i,j} - 4 b_{i+1,j} + b_{i+2,j}. \label{eq:d_formula}
\end{equation}

If $i = j$, then $b_{i-2,i} = b_{i,i-2} = (2i^3 - 6i^2 + i +6)/6$, $b_{i-1,i} = b_{i,i-1} = (2i^3 - 3i^2 + i)/6$, $b_{ii} = (2i^3 + i)/6$, $b_{i+1,i} = (2i^3 + 3i^2 + i)/6$, and $b_{i+2,i} =(2i^3 + 6i^2 + i)/6$, yielding $d_{i,i} = 1$.

For $i \neq j$, we consider several cases.
\begin{enumerate}
   	\item $j \le i-2$; Then $b_{i-2,j} = (3ij^2 + j  - 6j^2 - j^3)/6$, $b_{i-1,j} = (3ij^2 + j  - 3j^2 - j^3)/6$, $b_{i,j} = (3ij^2 + j - j^3)/6$, $b_{i+1,j} = (3ij^2 + j + 3j^2 - j^3)/6$, $b_{i+2,j} = (3ij^2 + j + 6j^2 - j^3)/6$, yielding $d_{i,j} = 0$.
   	
   	\item $j = i - 1$; Then $b_{i-2,j} = (2i^3 - 9i^2 + 13i - 6)/6$, $b_{i-1,j} = (2i^3 - 6i^2 + 7i - 3)/6$, $b_{i,j} = (2i^3 - 3i^2 + i)/6$, $b_{i+1,j} = (2i^3 - 5i + 3)/6$, $b_{i+2,j} = (2i^3 + 3i^2 - 11i + 6)/6$, yielding $d_{ i,i-1}= 0$;
   	
   	\item $j = i+1$; Then $b_{i-2,j} = b_{j,i-2} = (2i^3 - 3i^2 - 11i + 18)/6$, $b_{i-1,j} = b_{j,i-1} = (2i^3  - 5i + 3)/6$, $b_{i,j} = b_{j,i} = (2i^3  + 3i^2 + i)/6$, $b_{i+1,j} = (2i^3  + 6i^2 + 7i+3)/6$, $b_{i+2,j} = (2i^3  + 9i^2 + 13i+6)/6$, yielding $d_{ i,i+1} = 0$.
   	
   	\item $j \ge i-1$; Then $b_{i-2,j} = b_{j,i-2} = (3j(i-2)^2 + (i-2) - (i-2)^3)/6$, $b_{i-1,j} = b_{j,i-1} = (3j(i-1)^2 + (i-1) - (i-1)^3)/6$, $b_{i,j} = b_{j,i} = \frac{3ji^2 + i - i^3}{6}$, $b_{i+1,j} =  (3j(i+1)^2 + (i+1) - (i+1)^3)/6$,  $b_{i+2,j} = (3j(i+2)^2 + (i+2) - (i+2)^3)/6$, yielding $d_{ i,j}= 0$.
\end{enumerate}
   	   
\item[(ii)] The case $i = 1$. 

		For $j = 1$, $b_{1,1} = 3/6$, $b_{2,1} = 1$, and $b_{3,1}=3/2$; Thus, $d_{i,j}=d_{ 1,1} = 7 b_{1,1} - 4 b_{2,1} + b_{3,1} = 1$. 
		
		For $j > 1$, we have $b_{1,j} = b_{j,1} = j/2$, $b_{2,j} = b_{j,2} = 2j-1$, and $b_{3,j} = b_{j,3} = \frac{9j}{2} - 4$; Thus, $d_{i,j} = 7b_{1,j} - 4 b_{2,j} + b_{3,j} = 0$.
		   	   
   	   \item[(iii)] The case $i = 2$, with ${d_{i,j}} = -4b_{1,j} = 6 b_{2,j} - 4b_{3,j} + b_{4,j}$. 
   	   
   	   For $j = 2$, we have $b_{1,2} = b_{2,1} = 1$, $b_{2,2} = 3$, $b_{3,2} = 5$, $b_{4,2} = 7$; Thus, $d_{2,2} = 1$. 
   	   
   	   For $j \neq 2$, then $b_{1,j} = j/2$, $b_{2,j} = 2j-1$, $b_{3,j} = \frac{9j}{2} - 4$, and $b_{4,j} = 8j-10$. We have $d_{i,j} = 0$.
   	   \item[(iv)] For the case $i \in \{n-1,n\}$, similar computations using (\ref{eq:d_formula}) complete the proof.
\end{enumerate}
\end{proof}
From now on, we shall use $a_{i,j}^{-1}$ to denote the $(i,j)$-entry of $A^{-1}$, the inverse of $A$; thus, $a^{-1}_{i,j} = b_{i,j}$.
 
The following corollary is a consequence of Theorem~\ref{theo01}.
\begin{corollary} \label{cor1}
    The inverse of $A$ is a positive matrix; i.e., $A^{-1} > 0$, implying $a^{-1}_{i,j} > 0$.
\end{corollary}
\begin{proof}
  By Theorem \ref{theo01} it follows that $a^{-1}_{n,n}=  n(2n^2+1)/12$ is positive. Notice that, for $i \ge j$, $a_{i,j}^{-1} = \displaystyle \frac{3ij^2 + j - j^3}{6} \ge  \frac{3j^3 + j - j^3}{6} > 0$. Consequently, entries determined by the other 2 parts of Theorem \ref{theo01} are also positive.
\end{proof}

The above positivity result is important in the context of the fixed-point iteration we devise to solve the nonlinear system~\eqref{eq5}. Consider the iteration
\begin{eqnarray}
   \boldsymbol{u}^{\ell} = h^4 K A^{-1} \exp(-\boldsymbol{u}^{\ell-1}),  \ell = 1,2. \dots \label{fixed1}
\end{eqnarray}
Since $A^{-1} > 0$ (Corollary~\ref{cor1}), the recipe~\eqref{fixed1} generates a sequence of positive vectors $\{\boldsymbol{u}^{\ell}\}$, if started with   $\boldsymbol{u}^0 >  \boldsymbol{0}$. As the solution of this type of boundary-value problem is a nonnegative function (c.f., Section~\ref{section.intro}; see also later for the finite-difference equation case), if the above iteration converges, it converges to a positive solution.
    
Let $p \in \{1,2,\infty\}$. Our starting point for the convergence analysis is the relation, with $\boldsymbol{u}^0 > \boldsymbol{0}$,
\begin{equation}
\begin{aligned}
     \|\boldsymbol{u}^{\ell} - \boldsymbol{u}^{\ell-1}\|_p &= \|h^4 KA^{-1}(\exp(-\boldsymbol{u}^{\ell-1})-\exp(-\boldsymbol{u}^{\ell-2}))\|_p \notag \\
     &= h^4 K \| A^{-1} \left(\exp(-\boldsymbol{u}^{\ell-2}) + G(\boldsymbol{u}^{\ell-1}-\boldsymbol{u}^{\ell-2}) - \exp(-\boldsymbol{u}^{\ell-2})\right) \|_p \notag \\
     &= h^4 K \| A^{-1}  G(\boldsymbol{u}^{\ell-1}-\boldsymbol{u}^{\ell-2}) \|_p, \notag
\end{aligned}
\end{equation}
%We have the following expansion around  $\boldsymbol{u}^{\ell-2}$:
%$$
%   \exp(-\boldsymbol{u}^{\ell-1}) = \exp(\boldsymbol{u}^{\ell-2}) +  \nabla (\exp(-\boldsymbol{\xi}))(\boldsymbol{u}^{\ell-1} - \boldsymbol{u}^{\ell-2});
%$$ So $G = \nabla(\exp(-\boldsymbol{\xi}))$
where $G = -\text{diag}(\exp (-\xi_1),\dots,\exp(-\xi_n))$, such that the vector $\boldsymbol{\xi} = [\xi_i]_{i=1,n}  \in \mathcal{B} = \{ \boldsymbol{x} \in \mathbb{R}^n : \| \boldsymbol{x} - \boldsymbol{u}^{\ell-2} \|_p < \| \boldsymbol{u}^{\ell-1} - \boldsymbol{u}^{\ell-2} \|_p \}$. Since $\{\boldsymbol{u}^{\ell}\}$ is a sequence of positive vectors, $\boldsymbol{\xi}$ is also a positive vector, and consequently the diagonal entries of $G$ are strictly less than 1. Thus, $\|G\|_p < 1$, and
\begin{eqnarray}
   \|\boldsymbol{u}^{\ell} - \boldsymbol{u}^{\ell-1}\|_p &\le& h^4 K \| A^{-1} \|_p \|G\|_p  (\boldsymbol{u}^{\ell-1}-\boldsymbol{u}^{\ell-2}) \|_p \notag \\ 
   &<& h^4 K \| A^{-1} \|_p \| \boldsymbol{u}^{\ell-1}-\boldsymbol{u}^{\ell-2} \|_p. \label{eq:fixederror}
\end{eqnarray}
We define $L_p$ to be
\begin{equation}\label{eq:defLp}
L_p = h^4 K \|A^{-1}\|_p.
\end{equation}
Convergence guarantee of the fixed point iteration \eqref{fixed1} requires $L_p < 1$, which in turn, for given $K$ and chosen $h$, requires that 
\begin{equation}
\| A^{-1} \|_p < 1/(h^4K).
\end{equation}

\begin{lemma} \label{theo03}
For the inverse of $A$ in Theorem \ref{theo01}, the following holds true: 
\begin{equation*}
a_{i_1,j}^{-1} > a_{i_2,j}^{-1}, \quad \forall i_1 > i_2 > j, \text{ with } i_1,i_2,j \in \{1,2,\dots,n\}.
\end{equation*}
\end{lemma} 
\begin{proof}
From Theorem \ref{theo01} it follows that  $a^{-1}_{i_1,j} = (3i_1 j^2 + j - j^3)/6$ and  $a_{i_2,j}^{-1} = (3i_2 j^2 + j - j^3)/6$,  thus, one can notice that  $a_{i_1,j}^{-1} > a_{i_2,j}^{-1}$, for $i_1 > i_2 > j$.
\end{proof}

\begin{theorem} \label{theo04}
Let $A \in \mathbb{R}^{n \times n}$ be given in~\eqref{eq6}, with $n \ge 5$. Then
\begin{equation*}
	\|A^{-1}\|_p=
	\begin{cases*}
	(n^4 - n^2)/8, \quad \text{if }p=1, \\[0.5em]
	(n^4 + n^2)/8, \quad \text{if }p=\infty.
    \end{cases*}
\end{equation*}
%$\|A^{-1}\|_1 = \frac{1}{8}(n^4 - n^2)$ and $\|A^{-1}\|_{\infty} = \frac{1}{8}(n^4 + n^2)$.
\end{theorem}
\begin{proof}
For $p=1$ case, it follows from Lemma \ref{theo03} that
$$
  \| A^{-1} \|_1 = \max_{1 \le j \le n} \sum_{i = 1}^n | a_{i,j}^{-1} | = \max \left\{ \sum_{i=1}^n |a^{-1}_{i,n-1}|, \sum_{i=1}^n |a^{-1}_{i,n}|\right\}.
$$
We have
$$
  \sum_{i = 1}^n |a_{i,n-1}^{-1}| = \sum_{i=1}^n |a_{n-1,i}^{-1}| = \sum_{i=1}^n \frac{3(n-1)i^2 + i - i^3}{6} = \frac{n^4 - n^2}{8}. 
$$
We can now proceed similarly:
$$
  \sum_{i = 1}^n |a_{i,n}^{-1}| = \frac{1}{2}\sum_{i=1}^n  |a_{n,i}^{-1}| = \frac{1}{2} \sum_{i=1}^n \frac{3n i^2 + i - i^3}{6} =\frac{3n^4 + 4n^3 + 3n^2 + 2n}{48}.
$$
From the above results,
$$
  \sum_{i = 1}^n |a_{i,n-1}^{-1}| - \sum_{i = 1}^n |a_{i,n}^{-1}| = \frac{3n^4 - 4n^3 - 9n^2 - 2n}{48} > 0
$$
for $n \ge 5$. Therefore,
%\todo{Can we apply it here since inequality above is strictly?}
$$
   \max_{1 \le j \le n} \sum_{i=1}^n |a_{i,j}^{-1}| \le \sum_{i=1}^n |a_{i,n-1}^{-1}| = \frac{n^4 - n^2}{8} = \| A^{-1} \|_1.
$$

Next for $p=\infty$ using the Lemma \ref{theo03},
\begin{eqnarray}
  \|A^{-1}\|_{\infty} &=& \max_{1 \le i \le n} \sum_{j = 1}^n |a_{i,j}^{-1}| = \sum_{j = 1}^n |a_{n,j}^{-1}| = \sum_{j =1}^{n-1} |a_{n,j}^{-1}| + a_{n,n}^{-1} \notag \\
  &=& \sum_{j = 1}^{n-1} \frac{3nj^2 + j - j^3}{6} + \frac{n(2n^2+1)}{12} \notag \\
  &=& \frac{n^4 + n^2}{8}. \notag
\end{eqnarray}
\end{proof}   

% Proposed to tbe removed
%\begin{theorem}
%\sout{For the given $K > 0$ below, $n \ge 5$ and $L_p$ from (\ref{eq:defLp}), the fixed point iterations \eqref{fixed1} converge to a solution of \eqref{eq5} at the rate:}
%\begin{itemize} 
%\item $L_1 < K/8$, for $K \le 8$, independently of $h$;
%\item $L_2 < K/8$, for $K \le 8$, independently of $h$; %and
%\item $L_{\infty} < 17K/128$, for $K \le 128/17$. Furthermore, as $h \to 0$, $L_{\infty} < K/8$.
%\end{itemize}
%\end{theorem}
%\begin{proof}
%From Lemma \ref{theo04}, $\|A^{-1}\|_1 = \frac{n^4 - n^2}{8} < \frac{n^4}{8}$. Thus, with $h = 1/n$, $L_1 < h^4 K n^4/8 = K/8 < 1$, for $K \le 8$. 

%From Lemma \ref{theo04}, $\|A^{-1}\|_1 = (n^4 - n^2)/8 < n^4/8$. Thus, with $h = 1/n$, $$L_1 < h^4 K n^4/8 = K/8 < 1,$$ for $K \le 8$. 

%Also, $\|A^{-1}\|_{\infty} = \frac{n^4 + n^2}{8} < \frac{n^4 + n^4}{8} = n^4/4$. Thus, $L_{\infty} < h^4 K n^4/4 = K/4 < 1$, for $K \le 4$. 
%Also, $$L_{\infty} = h^4 K (n^4 + n^2)/8 = K (1 + 1/n^2)/8 < 17K/128 < 1,$$ for $K \le 128/17$. As $n \to \infty$, $L_{\infty} \to K/8$.

%Using H\"older's inequality $\|A^{-1}\|_2 \le \sqrt{\|A^{-1}\|_1 \|A^{-1}\|_{\infty}}$, we have $\sqrt{n^4 - 1} \le n^4/8,$$ which leads to the desired result.
%\end{proof}

Using H\"older's inequality, 
$$
\|A^{-1}\|_2 \le \sqrt{\|A^{-1}\|_1 \|A^{-1}\|_{\infty}} = \frac{1}{8}\sqrt{n^8 - n^4} \le \frac{1}{8}n^4.
$$

We conclude this  section by the characterization of the finite difference solution of the system~\eqref{eq5}.  Because $A \boldsymbol{u} = h^4 K \exp(-\boldsymbol{u}) > \boldsymbol{0}$, for the last row of the system,
\begin{eqnarray}
      2 u_{n-2} -4 u_{n-1} +2 u_n > 0 \quad \Longrightarrow \quad  u_n - u_{n-1} > u_{n-1} - u_{n-2}. \label{eq7}
\end{eqnarray}
From the $(n-1)$-th row, with $u_{n-3} - 4 u_{n-2} + 5 u_{n-1} - 2 u_n > 0$, we have
\begin{eqnarray}
        u_{n-1}-u_{n-2} > u_{n-2}-u_{n-3} + 2u_{n-2}-4 u_{n-1} + 2u_n > u_{n-2}-u_{n-3},
\end{eqnarray}
after using the inequality~\eqref{eq7}. Furthermore, this row leads to
$$
       4(u_{n-1} - u_{n-2}) > u_n - u_{n-3} + u_n - u_{n-1} > u_n - u_{n-3} + u_{n-1} - u_{n-2},
$$
after again using~\eqref{eq7}, which in turn yields
\begin{eqnarray}
       3 (u_{n-1}-u_{n-2}) > u_n - u_{n-3}.
\end{eqnarray}
    
We then have the following lemma:
    
\begin{lemma} \label{th01}
   For the inequality $A \boldsymbol{u} > \mathbf{0}$, with $A$ given by \eqref{eq6}, the following inequalities hold, with $j = i + 2$  and $i = 3\dots,n-1$ the rows of $A$:
   $$
   u_{j+2} - u_{j+1} > u_{j+1} - u_{j}, \quad 3(u_{j+2} - u_{j+1}) > u_{j+3} - u_j.
   $$
\end{lemma}

\begin{proof}
   We have proved the inequalities for $j = n-3$, which comes from the $(n-1)$-th row of $A \boldsymbol{u} > \mathbf{0}$. Now suppose that they hold also for $j = n-3, n-4, \dots, k+1$. Associated with $j=k$ is the inequality $u_k - 4 u_{k+1} + 6 u_{k+2} - 4 u_{k+3} + u_{k+4} > 0$ from the $(k+2)$-th row of $A \boldsymbol{u} > \mathbf{0}$, which gives
   \begin{eqnarray}
       u_{k+2}-u_{k+1} &>& 3u_{k+1} - u_k - 5 u_{k+2} + 4 u_{k+3} - u_{k+4} \notag \\
       &=& u_{k+1} - u_k + \left[ 4(u_{k+3} - u_{k+2}) + u_{k+1} - u_{k+4} + u_{k+1} - u_{k+2} \right] \notag \\
       &>& u_{k+1} - u_k + \left[ 3(u_{k+3} - u_{k+2}) + u_{k+1} - u_{k+4} \right] \notag \\
       &>& u_{k+1} - u_k \notag
   \end{eqnarray}
   by assumption. Next, note that
   $u_k - 4 u_{k+1} + 6 u_{k+2} - 4 u_{k+3} + u_{k+4} = 3(u_{k+2}-u_{k+1}) +u_k - u_{k+3} - [3(u_{k+3}-u_{k+2}) +u_{k+1}-u_{k+4}] > 0$. Thus, $3(u_{k+2}-u_{k+1}) +u_k - u_{k+3} >  3(u_{k+3}-u_{k+2}) +u_{k+1}-u_{k+4} > 0$, by assumption.
\end{proof}

\begin{theorem} \label{theo7}   
The solution of the finite difference system~\eqref{eq5} is a nonnegative vector $\boldsymbol{u}$, with increasing $u_i$.
\end{theorem} 
\begin{proof}
On the nodes $i =0,1$, approximation to the differential term leads to
    $$
       u_{i-2} - 4 u_{i-1} + 6 u_i - 4 u_{i+1} + u_{i+2} > 0,
    $$
which is of the same structure as the $i = 3,\dots,n-2$ rows of $A$. By Lemma~\ref{th01},
    $$
      u_2 - u_1 > u_1 - u_0, \quad u_1 - u_0 > u_0 - u_{-1}.
    $$
Therefore,
    $$
      u_{n} - u_{n-1} > u_{n-1} - u_{n-2} > \cdots > u_2 - u_1 > u_1 - u_0 > u_0 - u_{-1}.
    $$
With $u_0 =0$ (from the boundary condition $\phi(0) = \phi_0 = 0)$ and $u_{-1} = u_1$ (from using central finite differencing on $\phi'(0) = 0$), from the most right inequality, we get $u_1 > 0 = u_0$. Also, $u_2 - u_1 > u_1 - u_0 > 0$; thus $u_2 > u_1$. In general, we have $u_{i+1} > u_i$, $i = 1,\dots,n-1$.
\end{proof}
\section{The case with clamped-clamped boundary conditions} \label{section.CC}
In this section, we consider the case with the boundary conditions~\eqref{eq3}. Conditions at $x = 1$ are treated in the same way as at $x = 0$, leading to \eqref{eq5}, but now with $\boldsymbol{u} = (u_1, \dots, u_{n-1})^T \in \mathbb{R}^{n-1}$ and $A \in \mathbb{R}^{(n-1) \times (n-1)}$ given by
\begin{eqnarray}
A = \left[ \begin{array}{rrrrrrr} %{ccccccccccc}
 7       &   -4     &    1         &      0        &      &   \cdots & 0 \\
-4       &    6     &   -4         &     \ddots & \ddots &   &  \vdots  \\
 1       &   -4     &   \ddots   &    \ddots  &      & \ddots &   \\
 0       & \ddots & \ddots &   &  \ddots   &  \ddots & 0  \\ 
          &   \ddots   &           &                & \ddots   & -4 & 1 \\      
\vdots &          & \ddots &   \ddots  & -4 &  6 & -4 \\
0        & \cdots&            &    0       &   1   & -4 & 7
\end{array}\right]. \label{cc1}
\end{eqnarray}
However, to simplify our notation, we shall consider the case where $\boldsymbol{u} \in \mathbb{R}^n$ and $A \in \mathbb{R}^{n \times n}$ in the subsequent analysis; in this case, $h = 1/(n+1)$.

In constrast to \eqref{eq6}, the matrix \eqref{cc1} is centrosymmetric and near Toeplitz. Furthermore, it admits the rank-2 decomposition as follows:
\begin{eqnarray}
  A = T^2 + UU^t,  \label{cc2}
\end{eqnarray}
where $T = \text{tridiag}_n(-1,2,-1)$ is an $n \times n$ tridiagonal symmetric Toeplitz matrix, and
\begin{eqnarray}
 U = \left[ \begin{array}{cc}
             \sqrt{2} & 0 \\
                0     & 0 \\
                \vdots & \vdots \\
                0     & \sqrt{2}
            \end{array} \right] \in \mathbb{R}^{n \times 2}. \label{cc2a}
\end{eqnarray}
$T$ is a symmetric M-matrix, with positive inverse given explicitly by~(see, e.g., \cite{Diele98L})
\begin{eqnarray}
[T^{-1}]_{ij} = \begin{cases}
\frac{j}{n+1}(n-(i-1)), &i \ge j, \\[0.5em]
\frac{i}{n+1}(n-(j-1)), &i < j.
\end{cases} \label{cc3}
\end{eqnarray}

%\begin{eqnarray}
%   [T^{-1}]_{ij} = \begin{cases}
%                       \frac{j}{n+1}(n-(i-1)), &i \ge j, \\
%                       \frac{i}{n+1}(n-(j-1)), &i < j.
%   \end{cases} \label{cc3}
%\end{eqnarray}

$A$ is symmetric positive definite because $T^2 = T^TT$ (and $UU^t$) is symmetric positive (semi) definite. The inverse of $A$ can be computed by applying the Sherman-Morrison formula on \eqref{cc2}:
\begin{align}
  A^{-1} &= T^{-2} - T^{-2 }U (I_2 + U^t T^{-2} U)^{-1} U^t T^{-2} \notag\\
         &= T^{-1} (I - T^{-1}U  (I_2 + U^t T^{-2} U)^{-1} U^t T^{-1}  ) T^{-1} \notag   \\
         &= T^{-t} (I - T^{-1}U  (I_2 + U^t T^{-2} U)^{-1}  (T^{-1} U)^t ) T^{-1}.     \label{AinvFac}
\end{align}
Because $A^{-1}$ is symmetric positive definite, the middle term on the right-hand side $I - T^{-1}U  (I_2 + U^T T^{-2} U)(T^{-1}U)^T$ is also symmetric positive definite. Rewriting \eqref{cc2} as
$$
A = T^2 + UU^t = T^t (I + T^{-1}U (T^{-1}U)^t)  T, 
$$
clearly
$$
  (I + T^{-1}U (T^{-1}U)^t)^{-1} = I - T^{-1}U  (I_2 + U^t T^{-2} U)^{-1}  (T^{-1} U)^t =: M.
$$

Note that, with \eqref{cc2a} and \eqref{cc3},
\begin{eqnarray}
T^{-1}U = \frac{\sqrt{2}}{n+1} \left[\begin{array}{cc}
n & 1 \\
n-1 & 2 \\
\vdots & \vdots \\
2   & n-1 \\
1   & n
\end{array} \right]. \label{cc4}
\end{eqnarray}
Direct computation yields
\begin{eqnarray}
I_2 + U^t T^{-2} U &=& \frac{2}{(n+1)^2}  \left[ \begin{array}{ccc} \displaystyle 
\frac{(n+1)^2}{2}  + \sum_{k = 1}^n k^2 & & \displaystyle \sum_{k=1}^n (n-(k-1))k \\ 
& & \\
\displaystyle \sum_{k=1}^n (n-(k-1))k & & \displaystyle \frac{(n+1)^2}{2} + \sum_{k = 1}^n k^2 
\end{array} \right] \notag \\
&=& \frac{1}{\gamma}  \left[ \begin{array}{ccc}
\displaystyle  \gamma + \tau & & \displaystyle  \gamma n - \tau \\
& & \\
\displaystyle  \gamma n - \tau & & \displaystyle  \gamma + \tau
\end{array} \right],  \notag
\end{eqnarray}
where $ \tau = \frac{2n^3 + 3n^2 + n}{6}$ and  $\gamma = \frac{(n+1)^2}{2}$.
Its inverse is given by
\begin{eqnarray}
(I_2 + U^T T^{-2} U)^{-1} = \frac{1}{\delta} \left[ \begin{array}{ccc}
\displaystyle  \gamma + \tau & & \displaystyle  -\gamma n + \tau \\
& & \\
\displaystyle  - \gamma n + \tau & & \displaystyle  \gamma + \tau
\end{array} \right],  \label{cc5}
\end{eqnarray}
where $\det(I_2 + U^T T^{-2} U) = \frac{1}{3}(n^2 + 2n + 3) > 0$ and $\delta = (n+1)(2\tau + \gamma (1-n)) = \frac{1}{6}(n+1)^2(n^2 + 2n + 3)$.

Let $M = [m_{ij}]_{i,j = 1,n}$. Using \eqref{cc4} and \eqref{cc5}, we have, for $i \neq j$,
\begin{eqnarray}
  m_{ij} &=& - \frac{2}{\delta (n+1)^2}[n-(i-1) \quad i] \left[ \begin{array}{ccc}
  \displaystyle  \gamma + \tau & & \displaystyle  -\gamma n + \tau \\
  & & \\
  \displaystyle  - \gamma n + \tau & & \displaystyle  \gamma + \tau
  \end{array} \right]
  \left[ \begin{array}{c}
   n-(j-1) \\
       j
  \end{array} \right] \notag \\
  &=& q_0(n) + q_1(n) (i + j) + q_{2}(n)ij, \notag
\end{eqnarray}
where
\begin{eqnarray}
  q_0(n) &=& - \frac{4n^2 + 8n + 6}{(n+1)(n^2+2n+3)}, \notag \\
  q_1(n) &=& \frac{6}{n^2 + 2n + 3}, \notag \\
  q_{2}(n) &=& - \frac{12}{(n+1)(n^2+ 2n + 3)}. \notag
\end{eqnarray}
One can verify that $m_{ij}$ change signs. Thus $M$ is not an M-matrix.

For $i = j$, 
\begin{eqnarray}
  m_{ii} &=& 1 - \frac{2}{\delta(n+1)^2}[n-(i-1) \quad i] \left[ \begin{array}{ccc}
  \displaystyle  \gamma + \tau & & \displaystyle  -\gamma n + \tau \\
  & & \\
  \displaystyle  - \gamma n + \tau & & \displaystyle  \gamma + \tau
  \end{array} \right]
  \left[ \begin{array}{c}
  n-(i-1) \\
  i
  \end{array} \right] \notag \\
  &=&  \frac{n^3 - n^2 - 3n - 3}{(n+1)(n^2 + 2n + 3)} + \frac{12}{n^2 + 2n + 3} i - \frac{12}{(n+1)(n^2 + 2n + 3)}i^2 \notag \\
  &>& 0, \notag
\end{eqnarray}
for $n \ge 1$.

\begin{theorem} \label{AccPositive}
The inverse of $A$ given by~\eqref{cc1} is a positive matrix. Furthermore, let $\alpha = n+1-i$, $\beta = j \alpha/(6(n+1)(n^2 + 2n + 3))$, and $\varepsilon = 3(1 + \alpha (n+1))(1 + (i-j)j)$. The entries of $A^{-1}$ are
\begin{itemize}
  \item $a^{-1}_{ij} = \beta (\varepsilon + (j^2-1)(2\alpha^2+1))$, \text{for} $i \ge j$
  \item $a^{-1}_{ij} = a^{-1}_{ji}$, \text{otherwise}.
\end{itemize}
\end{theorem}

\begin{proof}
Let $A^{-1}= [a^{-1}_{ij}]$, with $A^{-1} = T^{-1}MT^{-1}$. Denote by $\mathbf{y}_j = [y_{k,j}]_{k=1,n} = MT_j^{-1}$, the product of $M$ and the $j$-th column of $T^{-1}$. For $i\geq j$,
\begin{eqnarray}
  \mathbf{y}_j = \frac{1}{n+1}
\begin{bmatrix}
(n+1-j)(m_{1,1}+2m_{1,2}+\cdots+jm_{1,j})+j(m_{1,j+1}(n-j)+\cdots+m_{1,n}) \\
\vdots\\
(n+1-j)(m_{j,1}+2m_{j,2}+\cdots+jm_{j,j})+j(m_{j,j+1}(n-j)+\cdots+m_{j,n})\\
\vdots \\
(n+1-j)(m_{n,1}+2m_{n,2}+\cdots+jm_{n,j})+j(m_{n,j+1}(n-j)+\cdots+m_{n,n})
\end{bmatrix}. \notag
\end{eqnarray}
Using $m_{i,j}$, for $i \le j$, we have, with $m^*_{i,j} = q_0 + q_1 (i+j) + q_2 ij$,
\begin{eqnarray}
  y_{i,j} &=& \frac{(n-(j-1))i}{n+1} \notag \\
          &+& \frac{n+1-j}{n+1}(m^{*}_{i,1}+2m^{*}_{i,2}+ \cdots + jm^{*}_{i,j})+\frac{j}{n+1}(m^{*}_{i,j+1}(n-j)+\cdots+m^{*}_{i,n}) \notag \\
          &=& \frac{(n-(j-1))i}{n+1} \notag \\
          &+& \frac{q_{0}+q_{1}i}{n+1}((n+1-j)(1+ \cdots+j)+j(n-j+\cdots+1)) \notag \\
&+& \frac{q_{1}+q_{2}i}{n+1}((n+1-j)(1^{2}+ \cdots+j^{2}) +  j((n-j)(j+1)+(n-j-1)(j+2)+\cdots+n)) \notag \\
          &=& \frac{1}{n+1} \left( (n-(j-1))i + r_0 + r_1 i \right), \notag 
\end{eqnarray}
where
$$
  r_0 = -\frac{j(n+1-j))((n+1)(n+1-j)+1)}{n^{2}+2n+3}
$$
and
$$
 r_1 = \frac{j(n+1-j)(n+1-2j)}{n^2+2n+3}.
$$
Using similar calculation for $i > j$, we get
\begin{equation}
y_{i,j}= \frac{1}{n+1}
\begin{cases} 
  \displaystyle  r_0 + r_1 i + (n+1-j)i = r_{0}+\left(r_{1}- j \right)i + (n+1)i, &  i \leq j, \\
  \displaystyle  r_0 + r_1 i + (n+1-i)j = r_{0}+\left(r_{1}- i \right)i + (n+1)j, &   i > j;
\end{cases} \notag
\end{equation}
hence,
\begin{equation}
\mathbf{y}_j = \frac{1}{n+1}\left( r_{0}
    \begin{bmatrix}
    1 \\ \vdots\\1 \\1 \\ \vdots\\ 1
    \end{bmatrix}
    +\left(r_{1}-j\right)
    \begin{bmatrix}
    1 \\ \vdots \\ j-1\\j\\ \vdots\\ n
    \end{bmatrix}
    + (n+1)
    \begin{bmatrix}
    1 \\ \vdots \\ j-1 \\ j \\ \vdots\\ j
    \end{bmatrix}\right). \notag
\end{equation}
Consider the $i$-th row of $T^{-1}$: 
$$
T_{i}^{-t} = \frac{1}{n+1}\begin{bmatrix}
    n+1-i & 2(n+1-i) & \cdots & i(n+1-i) & i(n-i) & \cdots &i
    \end{bmatrix}.
$$
We have
\begin{eqnarray}
   a^{-1}_{ij} &=& T_i^{-t} M T_j^{-1} = T_i^{-t} \mathbf{y}_j \notag \\
               &=& r_0 \frac{i (n+1-i)}{2} + \left(r_1 - \frac{j}{n+1}\right) \frac{i(n+1-i)(n+1+i)}{6} \notag \\
  &+& \frac{j(n+1-i)(3i(n+1)+1-j^{2})}{6(n+1)}  \notag \\
               &=& \beta (\varepsilon+(j^{2}-1)(2\alpha^{2}+1)), \notag
\end{eqnarray}
where
\begin{eqnarray}
    \alpha &=& \displaystyle {n+1-i}, \notag \\
    \beta &=& \displaystyle \frac{j(n+1-i)}{6(n+1)(n^{2}+2n+3)}, \notag \\
    \varepsilon &=& \displaystyle {3(1+\alpha+n\alpha)(1+ij-j^{2})}.  \notag
\end{eqnarray}
Notice that $\alpha, \beta, \varepsilon > 0$, $\forall i,j = 1,\dots,n$. With $i \ge j$ and $j^2 > j^2 - 1$,
\begin{eqnarray}
a_{ij}^{-1} &>& \beta (j^{2}-1)(2\alpha^{2}+1) \notag \\
            &\ge& 0. \notag
\end{eqnarray}
\end{proof}

By Theorem~\ref{AccPositive}, starting from $\boldsymbol{u}^0 > \boldsymbol{0}$, the fixed-point iteration~\eqref{fixed1} is guaranteed to generate a sequence of positive vectors. 

In the sequel, we present two ways of constructing an estimate for norms of the inverse of $A$. The first approach is based on the factorization $A^{-1} = T^{-t} M^{-1} T^{-1}$ in \eqref{AinvFac}. The result is presented in the next theorem.
\begin{theorem}
	For $p \in \{1,2,\infty\}$,
$$	
	\|A^{-1}\|_p \le (n+1)^4/32.
$$
\end{theorem}
\begin{proof}
	
	$$	
	\|A^{-1}\|_p \le \|T^{-1}\|_p \|M\|_p \|T^{-1}\|_p = \|T^{-1}\|_p^2 \|M \|_p.
	$$
Note that $\|T^{-1}\|_1 = \|T^{-1}\|_{\infty}$, due to symmetry. Thus, we shall consider only $\|T^{-1}\|_1$. Using \eqref{cc3},
\begin{eqnarray}
  \sum_{j = 1}^n |T^{-1}_{ij}| &=& \frac{1}{n+1} \left[ (n-(i-1)) \sum_{j=1}^{i-1} j + i \sum_{j=1}^{n-(i-1)}j
  \right] \notag \\
     &=& \frac{1}{2(n+1)} \left[ (n+1)^2 i - (n+1)i^2 \right]. \notag
\end{eqnarray}
The maximum of the rowsum is then attained for $i = (n+1)/2$. Thus,
\begin{eqnarray}
\|T^{-1}\|_1 = \max_{1 \le i \le n} \sum_{j=1}^n |T^{-1}_{ij}| \le \frac{(n+1)^2}{8},
\end{eqnarray}
with equality holding when $n$ is odd.

We now estimate the 1-norm of $M$. Let $\widetilde{m}_{ij} = q_0(n) + q_1(n)(i+j) + q_2(n) ij$, $\forall i,j = 1,\dots,n$ and consider $\sum_{j = 1}^n |\widetilde{m}_{ij}|$. For a fixed $i$, $\widetilde{m}_{ij}$ can be viewed as a linear function of $j$. $\sum_{j = 1}^n |\widetilde{m}_{ij}|$ can then be viewed as the rectangular rules that approximate the area made by the function $\widetilde{m}_{ij}$ and the $j$-axis. In this case, treating $j \in [0,n+1] \subset \mathbb{R}$,
$$
\sum_{j = 1}^n |\widetilde{m}_{ij}| \le \int_{j=0}^{n+1} |\widetilde{m}_{ij}|dj = \frac{1}{2} (|\widetilde{m}_{i,0}| + |\widetilde{m}_{i,n+1}| )(n+1),
$$
where $\widetilde{m}_{i,0} = - (4n^2 + 6n(1-i) + 8-6i)/[(n+1)(n^2 + 2n + 3)]$ and $\widetilde{m}_{i,n+1} =  (2n^2 + 6n -2 - 6i(n+1))/[(n+1)(n^2 + 2n + 3)]$.

Since the matrix $\widetilde{M} = [\widetilde{m}_{ij}]$ is persymmetric, we just need to consider $i = 1,\dots,(n+1)/2$. Then,
\begin{eqnarray}
\sum_{j = 1}^n |\widetilde{m}_{ij}| &\le& \max_i \frac{1}{2} (|\widetilde{m}_{i,0}| + |\widetilde{m}_{i,n+1}| )(n+1) = \frac{1}{2} \frac{2(n^2 + 5)}{(n+1)(n^2+2n+3)}(n+1) \notag \\
&=& \frac{n^2 + 5}{n^2+2n+3}. \notag
\end{eqnarray}
Now,
\begin{eqnarray}
  \sum_{j=1}^n |m_{ij}| &=& \sum_{j = 1,j\neq 1}^n |m_{ij}| + |m_{ii}| = \sum_{j = 1,j\neq 1}^n |\widetilde{m}_{ij}| + |1 + \widetilde{m}_{ii}|  \notag \\
  &\le& 1 + |\widetilde{m}_{ii}| + \sum_{j = 1,j \neq i}^n |\widetilde{m}_{ij}| \notag \\
  &=& 1 + \sum_{j = 1}^n |\widetilde{m}_{ij}|. \notag
\end{eqnarray}
Thus, for $n \ge 1$,
\begin{eqnarray}
\|M\|_1 = \max_i \sum_{j=1}^n |m_{ij}| &\le&  1 + \max_i \sum_{j = 1}^n |\widetilde{m}_{ij}|  \notag \\
&\le& 1 + \frac{n^2 + 5}{n^2+2n+3} \notag \\
&\le& 2, \notag
\end{eqnarray}
since $n^2 + 5 < n^2 + 2n + 3$ for $n \ge 1$.

Combining with $\|T^{-1}\|_1$, we get the desired result. Furthermore, using  H\"older's inequality, $\|A^{-1}\|_2 \le \sqrt{\|A^{-1}\|_1 \|A^{-1}\|_{\infty}} \le (n+1)^4/32$.
\end{proof}

The second approach uses the knowledge of the entries of $A^{-1}$ in Theorem~\ref{AccPositive}. Tedious calculation results in exact norms in some cases, and hence much stronger estimates than the previous estimates. 

\begin{theorem} \label{theo33}
For $p \in \{1,2,\infty\}$,
$$
  \|A^{-1}\|_p  \le  (n+1)^2 \left( (n+1)^2 + 8\right)/384. %\frac{(n+1)^2}{384} \left( (n+1)^2 + 8\right)
$$
If $n$ is odd, then the equality holds for $p \in \{1,\infty\}$.
\end{theorem}
\begin{proof}
We shall  first  consider the case $p = \infty$. In this case, by using $a_{i,j}^{-1} > 0$,
\begin{eqnarray}
   \| A^{-1} \|_{\infty} &=& \max_i \sum_{j=1}^n |a^{-1}_{i,j}| = \max_i \sum_{j=1}^n a^{-1}_{i,j} \notag 
\end{eqnarray}
For $i = 1,\dots, n$,
\begin{eqnarray}
   \sum_{j=1}^n a_{i,j}^{-1} = \sum_{j=1}^i a^{-1}_{i,j} + \sum_{j=i+1}^n a^{-1}_{i,j} = \sum_{j=1}^i a^{-1}_{i,j} + \sum_{k = 1}^{n-i} a^{-1}_{k,i}, \notag
\end{eqnarray}
because of the centrosymmetry of $A^{-1}$. Calculating each sum using the formula for the entries $a^{-1}_{ij}$, we get
\begin{eqnarray}
   \sum_{j=1}^i a^{-1}_{i,j} &=&\widehat{\delta}^{-1}  \left[ C^i_1 \sum_{j=1}^i j + C^i_2 \sum_{j=1}^i j^2 + C^i_3 \sum_{j=1}^i j^3 \right] \notag \\
    &=& \widehat{\delta}^{-1} \left[ C^i_1 \frac{i^2 + i}{2} + C^i_2 \frac{2i^3 + 3i^2 + i}{6}+ C^i_3 \frac{i^4 + 2i^3 + i^2}{4} \right], \notag
\end{eqnarray}
where $\widehat{\delta} = 6(n+1)(n^2 + 2n + 3)$ and
\begin{eqnarray}
   C^i_1 &=& n^3 + 3n^2 - 3i^2 n + 5n + 2i^3 - 3i^2 - 2i + 3, \notag \\
   C^i_2 &=& 3in^3 - 6i^2 n^2 + 9in^2 + 3i^3n -12i^2n + 12in + 3i^3 - 9i^2 + 6i, \notag \\
   C^i_3 &=& -n^3 - 3n^2 + 3i^2n - 5n - 2i^3 + 3i^2 + 2i - 3. \notag
\end{eqnarray}
Also,
\begin{eqnarray}
   \sum_{k=1}^{n-i} a^{-1}_{k,i} &=& \widehat{\delta}^{-1} \left[ C^k_1 \sum_{k=1}^{n-i} k + C^k_2 \sum_{k=1}^{n-i} k^2 + C^k_3 \sum_{k=1}^{n-i} k^3 \right] \notag \\
    &=& \widehat{\delta}^{-1} \left[  C^k_1 \frac{(n-i)^2 + (n-i)}{2} + C^k_2 \frac{2(n-i)^3 + 3(n-i)^2 +(n-i)}{6} \right] \notag \\
    &+& \widehat{\delta}^{-1} C^k_3 \frac{(n-i)^4 + 2(n-i)^3 +(n-i)^2}{4}, \notag
\end{eqnarray}
where
\begin{eqnarray}
   C^k_1 &=& 3i^2n - 2i^3 + 3i^2 + 2i, \notag \\
   C^k_2 &=& 3i^2n^2 - 3i^3n + 6i^2 n + 3in - 3i^3 + 3i, \notag \\
   C^k_3 &=& -3i^2 n + 2i^3 - 3i^2 - 2i. \notag
\end{eqnarray}

Assuming that $i \in [1,n] \subset \mathbb{R}$, the maximum of the rowsum is obtained from the condition $\displaystyle \frac{d}{di}\sum_{j=1}^n a_{i,j}^{-1} = 0$.  In this regard, we have
\begin{eqnarray}
   \frac{d}{di}\sum_{j=1}^n a_{i,j}^{-1} = \widehat{\delta}^{-1} \left[ C'_0 + C'_1 i + C'_2 i^2 + C'_3 i^3 \right] = 0, \notag
\end{eqnarray}
where
\begin{eqnarray}
    C'_0 &=& \frac{1}{2}n^4 + 2n^3 + 4n^2 + 4n + \frac{3}{2}, \notag \\
    C'_1 &=& \frac{1}{2}n^5 + \frac{5}{2}n^4 + 5n^3 + 5n^2 + \frac{1}{2}n - \frac{3}{2}, \notag \\
    C'_2 &=& -\frac{3}{2}n^4 - 6n^3 - 12n^2 - 12n - \frac{9}{2}, \notag \\
    C'_3 &=& n^3 + 3n^2 + 5n + 3. \notag
\end{eqnarray}
The only acceptable solution of the above equation is $i = (n+1)/2$. The other solutions are rejected: $i = - \frac{1}{2}(\sqrt{n^2 + 2n + 5} - (n+1)) < 0$ and $i = \frac{1}{2} (\sqrt{n^2 + 2n + 5} + (n+1)) > n+1 > n$. One can verify that $i = (n+1)/2$ maximizes the rowsum.

Let $n$ be odd. With $i = (n+1)/2$,
\begin{equation}
\begin{aligned}
   \| A^{-1} \|_{\infty} &=\displaystyle \max_i \sum_{j=1}^n |a^{-1}_{i,j}| = \sum_{j=1}^n a^{-1}_{(n+1)/2,j} = \sum_{j=1}^{\frac{n+1}{2}} a^{-1}_{(n+1)/2,j} + \sum_{k = 1}^{\frac{n-1}{2}} a^{-1}_{k,(n+1)/2} \notag  \\
  &=  \left( n^4 + 4n^3 + 14n^2 + 20n + 9\right)/384 \notag \\
 &=  (n+1)^2 ( (n+1)^2 + 8)/384.
\end{aligned}
\end{equation}
If $n$ is even, then $i = (n+1)/2$ is not a row of the matrix $A$; the maximum of the rowsum will then be attained at $i = \ceil{(n+1)/2}$ or $i = \floor{(n+1)/2}$. Either case satisfies
\begin{eqnarray}
  \| A^{-1} \|_{\infty} \le(n+1)^2 ( (n+1)^2 + 8)/384. \notag
\end{eqnarray}
Symmetry of $A^{-1}$ leads to $\| A^{-1}\|_1 = \|A^{-1}\|_{\infty}$. Using H\"older's inequality, the above inequality holds also for $p = 2$.
\end{proof}
% \section{Numerical Results} \label{section.num}
Table~\ref{table5} shows the computed norms of the inverse and compares them with the estimate given by Theorem~\ref{theo33}. For odd $n$ and $p \in \{1,\infty\}$ the norms are exact. For even $n$, Theorem~\ref{theo33} gives  an estimate that leads to a small gap. This gap relative to the estimate becomes negligible with an increase in $n$. To support this statement, the reader is referred to Fig. \ref{fig:norm_p_vs_Upper_Bound} and Fig. \ref{fig:relative_error_p} in log scales. The numerical tests are performed for all even $n$ from $10$ to $1000$. The relative error is computed as $|\|A^{-1}\|_p-UBound|/\|A^{-1}\|_p$, where $UBound=(n+1)^2 \left( (n+1)^2 + 8\right)/384$ from Theorem~\ref{theo33}. As shown in Fig. \ref{fig:relative_error_p} (left), the relative error decreases as $n$ increases for $p=1$ or $p=\infty$. On the other hand, according to the numerical observation the difference between $\| A\|_2$ and the upper bound become constant  relative to the norm as $n$ increases, see Fig. \ref{fig:relative_error_p} (right).
\begin{table}[!h]
	\caption{Computed $\|A^{-1}\|_p$ and the estimates, for the clamped-clamped case.}  \label{table5}
	\begin{center}
		\begin{tabular}{c|ccc|c} \hline
			$n$ & \multicolumn{3}{c|}{$p=$}& Upper bound from\\ \cline{2-4}
			    & 1 & 2 & $\infty$ &  Theorem~\ref{theo33} \\ \hline
                   49  & 16,328 &  12,527 & 16,328 & 16,328 \\
			50  & 17,658 & 13,558 & 17,658 & 17,672 \\ \hline
			99  & 260,625 & 199,939 & 260,625 & 260,625 \\ 
                  100 & 271,150 & 208,055  & 271,150 & 271,203 \\ \hline
                  150 & 1,354,225 & 1,038,976 & 1,354,225 & 1,354,343 \\ \hline
		\end{tabular}
    \end{center}
\end{table}
\begin{figure}[H]
	\centering
	\includegraphics[width=0.4\linewidth]{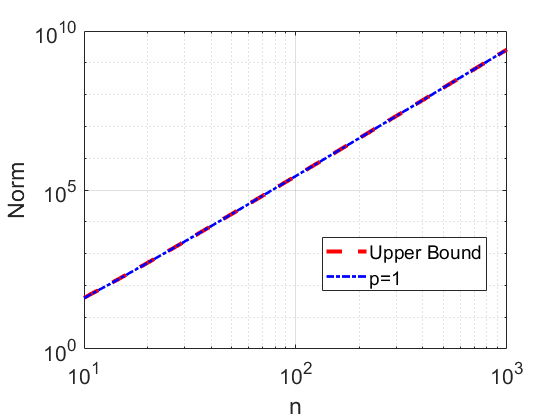}
	\includegraphics[width=0.4\linewidth]{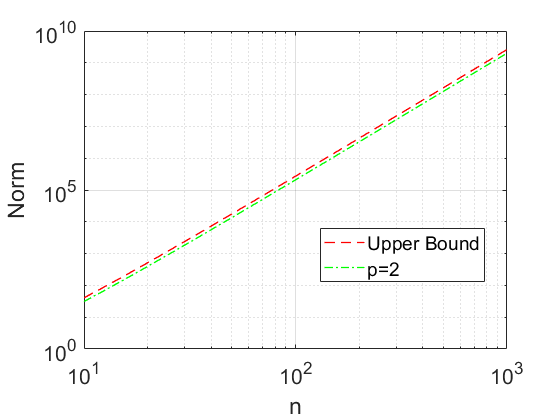}
	\caption{The upper bound and actual norm $p=1$ or $p=\infty$(left) and $p=2$(right) in log scale}
	\label{fig:norm_p_vs_Upper_Bound}
\end{figure}
\begin{figure}[H]
	\centering
	\begin{subfigure}[b]{0.4\textwidth}
    \centering
	\includegraphics[width=0.9\linewidth]{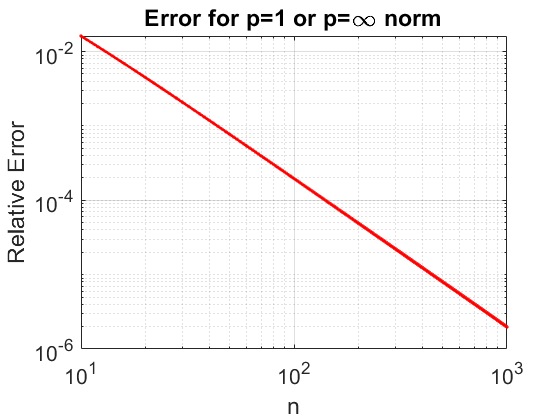}
	\caption{$p=1$ or $p=\infty$(left)}
	\end{subfigure}
	\hfill
	\begin{subfigure}[b]{0.4\textwidth}
    \centering
	\includegraphics[width=0.9\linewidth]{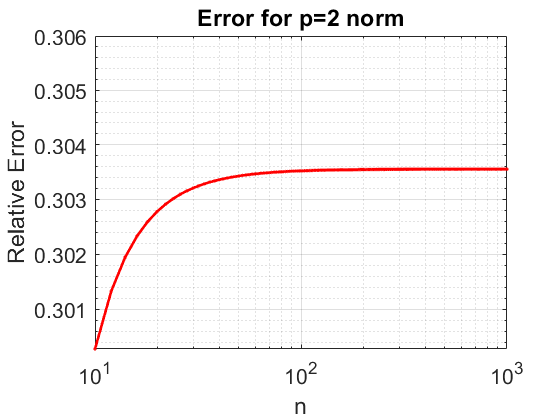}
	 \caption{$p=2$(right)}
     \end{subfigure}
	\caption{The relative errors in log scale}
	\label{fig:relative_error_p}
\end{figure}

%\begin{theorem}
%Let $n \ge 5$, and $A$ be given as in \eqref{cc1}. The fixed point iteration \eqref{fixed1} converges to a solution at the rate 
%$$
%L_p \le \frac{K}{384} \left( 1 + \frac{8}{(n+1)^2}  \right), \quad p \in \{1,2,\infty\}.
%$$
%for $K < 384(n+1)^2/((n+1)^2 + 8)$.
%\end{theorem}

%\begin{proof}
%\begin{eqnarray}
%L_p = K h^4 \|A^{-1}\|_p &\le& K \frac{1}{(n+1)^4} \frac{(n+1)^2}{384} ( (n+1)^2 + 8) \notag \\
%&=& \displaystyle \frac{K}{384} \left( 1 + \frac{8}{(n+1)^2}  \right).
%\end{eqnarray}
%Forcing $L_p < 1$ leads to the theorem.
%\end{proof}

For $n \ge 5$, the factor $(1 + 8/(n+1)^2)/384 \le 11/3474$. So, alternatively, if $K$ satisfies the condition in the above theorem, we can have a simpler bound: $L_p < 11/3474$. This factor approaches $1/384$ from above as $n \to \infty$. Since the latter is slightly less than the former, for a fixed $K$, one can expect a slight improvement of convergence by increasing $n$. 
%\ya{Since $K$ depends on the physical parameters}

\section{Numerical Results}  \label{section.num}

In this section, we present numerical results from solving~\eqref{eq1} with~\eqref{eq2} or~\eqref{eq3} using the fixed point method\eqref{fixed1}. We compare the observed convergence with the theoretical bound given by~\eqref{eq:defLp} and Theorem~\ref{theo04} (for the clamped-free case) or Theorem~\ref{theo33} (for the clamped-clamped case).The fixed point method \eqref{fixed1} is declared to have reached a convergence if $\|\boldsymbol{u}^{\ell+1} - \boldsymbol{u}^{\ell}\|_p < 10^{-6}$, where $p \in \{1,2,\infty\}$. Solutions at convergence are shown in Figure~\ref{fig1} for the clamped-free and clamped-clamped case, with $K = 1$.

% \begin{figure}[H]
% \includegraphics[width=0.5\textwidth]{images/clamped_free_K_1.png} \includegraphics[width=0.5\textwidth]{images/clamped_clamped_K_1.png}
% \caption{Solution at convergence with $K = 1$, $n = 100$: clamped-free (left), clamped-clamped (right)} \label{fig1}
% \end{figure}

\begin{figure}[H]
\includegraphics[width=0.5\textwidth]{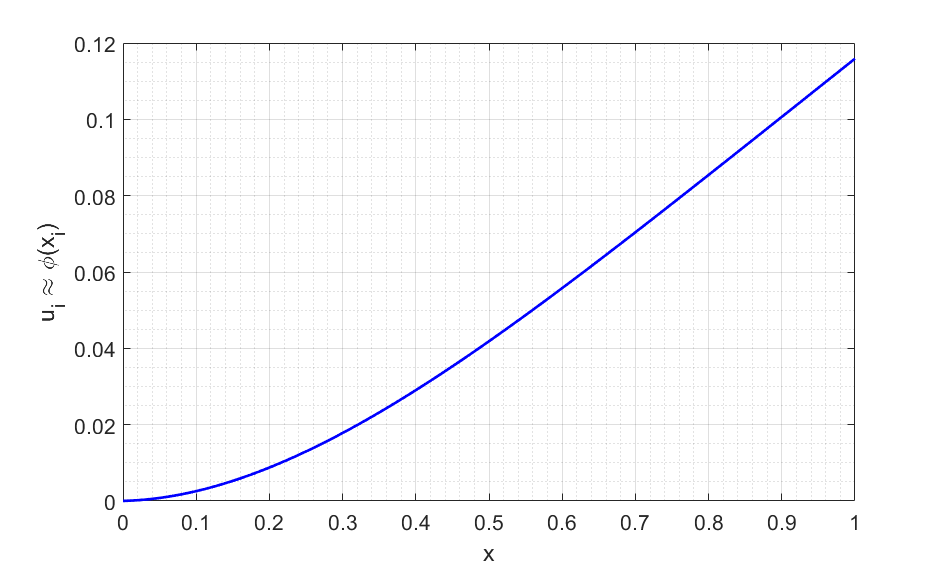} \includegraphics[width=0.5\textwidth]{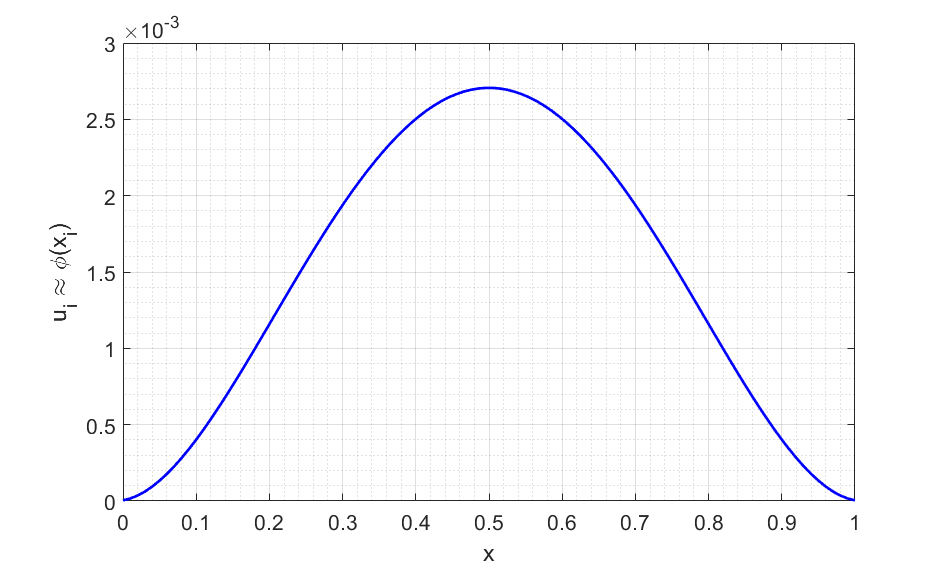}
\caption{Solution at convergence with $K = 1$, $n = 100$: clamped-free (left), clamped-clamped (right)} \label{fig1}
\end{figure}

For both cases, the actual convergence rates are lower than the estimate (Tables \ref{table1}--\ref{table4}), with increasing gaps between the two as $K$ increases. As $\|A^{-1}\|_p$ is exact, except for $p = 2$, (due to the explicit inverse of $A$), this suggests that the gap in the convergence rate is mainly due to the estimate $\|G\|_p < 1$.  The numerical experiments suggest that the simple fixed-point method \eqref{fixed1} can be used for a wider range of $K$ than suggested by the theoretical results. For instance, with $K = 386$ and $n = 99$, we have $L_p = 1.006$. The method still however converges to the solution at the maximum rate of $0.5278$.
\begin{table}[H] 
	\caption{Observed maximum convergence rate for clamped-free case, with $n = 50$. In brackets are the theoretical rate based on Theorem~\ref{theo04}.} \label{table1}
	\begin{center}	
		\begin{tabular}{c|cccccccc}
			\hline
			& \multicolumn{8}{c}{$p = $} \\ \cline{2-9}
			$K$  & \multicolumn{2}{c}{$1$} && \multicolumn{2}{c}{$2$} && \multicolumn{2}{c}{$\infty$}  \\ \hline
			$1/8$  & 0.010 & [0.016] && 0.010 & [0.016] && 0.010 & [0.017] \\
			1      & 0.074 & [0.125] && 0.074 & [0.125] && 0.074 & [0.125] \\
			8      & 0.400 & [1.000] && 0.400 & [1.000] && 0.402 & [1.000] \\ \hline
		\end{tabular}
	\end{center}
\end{table}

%\begin{table}[!h] 
%	\caption{Observed maximum convergence rate for clamped-free case, with $n = 50$. In brackets are the theoretical rate.} \label{table1}
%  \begin{center}	
%	\begin{tabular}{c|cccccccc}
%		\hline
%		   & \multicolumn{8}{c}{$p = $} \\ \cline{2-9}
%	  $K$  & \multicolumn{2}{c}{$1$} && \multicolumn{2}{c}{$2$} && \multicolumn{2}{c}{$\infty$}  \\ \hline
%	  $1/8$  & 0.010 & [0.016] && 0.010 & [0.016] && 0.010 & [0.017] \\
%	  1      & 0.074 & [0.125] && 0.074 & [0.125] && 0.074 & [0.125] \\
%	  8      & 0.400 & [1.000] && 0.400 & [1.000] && 0.402 & [1.000] \\ \hline
%	\end{tabular}
%   \end{center}
%\end{table}

\begin{table}[H] 
	\caption{Observed maximum convergence rate for clamped-free case, with $n = 99$. In brackets are the theoretical rate based on Theorem~\ref{theo04}.} \label{table2}
	\begin{center}
		\begin{tabular}{c|cccccccc}
			\hline
			& \multicolumn{8}{c}{$p = $} \\ \cline{2-9}
			$K$ & \multicolumn{2}{c}{$1$} && \multicolumn{2}{c}{$2$} && \multicolumn{2}{c}{$\infty$}  \\ \hline
			$1/8$  & 0.010 & [0.016] && 0.010 & [0.016] && 0.010 & [0.017] \\
			1      & 0.074 & [0.125] && 0.074 & [0.125] && 0.074 & [0.125] \\
			8      & 0.400 & [1.000] && 0.400 & [1.000] && 0.402 & [1.000] \\ \hline
		\end{tabular}
	\end{center}
\end{table}

\begin{table}[H] 
	\caption{Observed maximum convergence rate for the clamped-clamped case, with $n = 49$. In brackets are the theoretical rate based on Theorem~\ref{theo33}.} \label{table3}
	\begin{center}	
		\begin{tabular}{c|cccccccc}
			\hline
			& \multicolumn{8}{c}{$p = $} \\ \cline{2-9}
			$K$  & \multicolumn{2}{c}{$1$} && \multicolumn{2}{c}{$2$} && \multicolumn{2}{c}{$\infty$}  \\ \hline
			$1/8$  & 0.0003 & [0.0033] && 0.0003 & [0.00033] && 0.0003 & [0.00033] \\
			1      & 0.0020 & [0.0026] && 0.0020 & [0.0026] && 0.0020 & [0.0026] \\
			8      & 0.0158 & [0.0209] && 0.0159 & [0.0209] && 0.0161 & [0.0209]  \\ 
			32     & 0.0615 & [0.0836] && 0.0619 & [0.0836] && 0.0627 & [0.0836] \\ 
                 128     & 0.2223 & [0.3344]  && 0.2237 & [0.3344] && 0.2262 & [0.3344] \\  \hline
		\end{tabular}
	\end{center}
\end{table}

\begin{table}[H] 
	\caption{Observed maximum convergence rate for the clamped-clamped case, with $n = 100$. In brackets are the theoretical rate based on Theorem~\ref{theo33}} \label{table4}
	\begin{center}	
		\begin{tabular}{c|cccccccc}
			\hline
			& \multicolumn{8}{c}{$p = $} \\ \cline{2-9}
			$K$  & \multicolumn{2}{c}{$1$} && \multicolumn{2}{c}{$2$} && \multicolumn{2}{c}{$\infty$}  \\ \hline
			$1/8$  & 0.0002 & [0.00033] && 0.0002 & [0.00033] && 0.0003 & [0.00033] \\
			1      & 0.0020 & [0.0026] && 0.0020 & [0.0026] && 0.0020 & [0.0026] \\
			8      & 0.0157 & [0.0208] && 0.0159 & [0.0208] && 0.0160 & [0.0208]  \\
                  32     & 0.0614 & [0.0834] && 0.0618 & [0.0834] && 0.0625 & [0.0834] \\
                128     & 0.2218 & [0.3336]  && 0.2232 & [0.3336] && 0.2257 & [0.3336] \\  \hline
		\end{tabular}
	\end{center}
\end{table}

%\newpage %this can be removed 
\section{Conclusion}  \label{section.conclusion}

The explicit inverse formula for pentadiagonal matrices arising in the fourth-order nonlinear beam boundary value problem were constructed. The explicit formula helped computing some norms of their inverse, used to estimate the convergence of a fixed-point iteration for solving the nonlinear system of equations. Further research on the convergence upper bounds is necessary to extend our knowledge of the convergence rate in the fixed point method.

\section*{Acknowledgment} 

BK and YA wishes to acknowledge the research grant, No AP08052762, from the Ministry of Education and Science of the Republic of Kazakhstan and the Nazarbayev University Faculty Development Competitive Research Grant (NUFDCRG), Grant No 110119FD4502.
% The first author thanks P. Skrzypacz for introducing him to the problem discussed in  this paper.

%We are thankful for discussion with Professor Todd Arbogast. 
\bibliographystyle{unsrt}  %plainnat unsrt elsarticle-num
\bibliography{ref}

\end{document}